\documentclass{amsart}[12pt]
\usepackage{amsmath}
\usepackage{amsfonts}
\usepackage{amssymb}
\usepackage{hyperref}
\usepackage{mathrsfs}
\usepackage{graphicx}
\usepackage{etex,seqsplit}
\usepackage[shortlabels]{enumitem}
\usepackage{eucal}
\usepackage{morefloats, comment}
\usepackage{thm-restate}
\usepackage[margin=1in]{geometry}

\usepackage[colorinlistoftodos]{todonotes}

\newcommand{\F}{\mathbb{F}}
\newcommand{\Q}{\mathbb{Q}}
\newcommand{\Z}{\mathbb{Z}}
\DeclareMathOperator{\gsp}{GSp}
\newcommand{\OO}{\mathscr{O}}
\newcommand{\p}{\mathfrak{p}}
\DeclareMathOperator\Sp{Sp}

\DeclareMathOperator\mult{mult}
\DeclareMathOperator\sgn{sgn}
\DeclareMathOperator\disc{disc}

\DeclareMathOperator\cyc{cyc}
\DeclareMathOperator\ab{ab}

\theoremstyle{plain}
\newtheorem{theorem}{Theorem}[section]

\newtheorem{lemma}[theorem]{Lemma}

\theoremstyle{definition}
\newtheorem{definition}[theorem]{Definition}

\newtheorem*{theorem*}{Theorem}
\newtheorem*{proposition*}{Proposition}
\newtheorem*{lemma*}{Lemma}

\theoremstyle{remark}

\numberwithin{equation}{section}

\begin{document}

\title{An explicit Abelian Surface with Maximal Galois Action}

\author{Quinn Greicius} \address{Department of Mathematics, Stanford University, \mbox{Stanford, CA 94305}} 
\email{qrg@stanford.edu}
\author{Aaron Landesman} \address{Department of Mathematics, Stanford University, \mbox{Stanford, CA 94305}}
\email{aaronlandesman@gmail.com}

\begin{abstract}
	We construct an explicit example of a genus $2$ curve $C$ over a number field $K$ such that the adelic Galois representation arising from the action of $\operatorname{Gal}(\overline{K}/K)$ on the Jacobian of $C$ has image $\gsp_4(\widehat{\mathbb Z})$.
\end{abstract}
\maketitle

\section{Introduction}
Let $K$ be a number field and $A$ a principally polarized abelian variety of dimension $g$ over $K$.
For $n$ a positive integer, the action of $G_K:=\mathrm{Gal}(\overline{K}/K)$ on the $n$-torsion $A[n]$ preserves the symplectic form given by the Weil pairing and yields 
the \textit{mod-n} Galois representation $$\rho_{A,n}\colon G_K\to \gsp_{2g}(\Z/n\Z).$$
The inverse limit of the $\rho_{A,n}$ over $\Z/m\Z \rightarrow \Z/n\Z$ for $n \mid m$ forms the \textit{adelic} Galois representation
$$\rho_{A}\colon G_K\to \gsp_{2g}(\widehat{\Z}).$$
For $\ell$ a prime, the \textit{$\ell$-adic} Galois representation
$$\rho_{A,\ell^\infty}\colon G_K\to\gsp_{2g}(\Z_\ell)$$
is the composition of $\rho_{A}$ with the map $\gsp_{2g}(\widehat{\Z}) \rightarrow \gsp_{2g}(\Z_\ell)$.

There has been much recent interest in understanding the image of Galois representations.
One of the earliest results in this direction is Serre's Open Image Theorem ~\cite{serre:open-image-theorem}, which states that for an elliptic curve $E/K$ without complex multiplication, $\rho_E(G_K)$ is an open subgroup of $\gsp_{2}(\widehat{\Z})$. Serre's subsequent generalization of this result in ~\cite[Theorem 3]{serrelacs} implies that for $A$ as above, with odd dimension $g$ (or dimension $g=2$ or $6$) and $\operatorname{End}(A)\cong \Z$, $\rho_A(G_K)$ is open in $\gsp_{2g}(\widehat{\Z})$. 
Note that $\rho_A(G_K)$ is open in $\gsp_{2g}(\widehat{\Z})$ if and only if $\rho_{A,\ell^\infty}(G_K)$ is open in $\gsp_{2g}(\Z_\ell)$ for all $\ell$ and equal to $\gsp_{2g}(\Z_\ell)$ for all but finitely many $\ell$. In the dimension $1$ case, however, despite the fact that the Galois representation has open image, it turns out that if $K = \mathbb Q$, $\rho_E$ can never surject onto $\gsp_{2}(\widehat{\mathbb Z})$
\cite[Proposition 22]{serre:open-image-theorem}.
Nevertheless, it is possible that $\rho_E(G_K) = \gsp_{2}(\widehat{\mathbb Z})$ in the case $K \neq \mathbb Q$, and in~\cite{greiciusthesis}, A. Greicius constructs an example of such an $E$.
Furthermore, in~\cite{zywina2015example}, Zywina constructs an example of a non-hyperelliptic curve of genus $3$ over $\mathbb Q$ whose Jacobian has adelic Galois image equal to $\gsp_6(\widehat{\mathbb Z})$. 
Hence, while we do have examples of curves $C$ in genus $g = 1$ and $3$, 
with $\rho_{J(C)}(G_K) = \gsp_{2g}(\widehat{\mathbb Z})$, to the authors' knowledge,
no such example is known in the case $g = 2$. Indeed, there turn out to be significant obstacles in
genus $2$ faced neither in genus $1$ nor genus $3$.
The purpose of this note is to provide an example of such a genus $2$ curve, given in \autoref{theorem:polynomial}.
	
    The techniques used in the genus $1$ and $3$ cases appear not
to apply in the genus $2$ case: the genus $1$ techniques of 
~\cite{greiciusthesis}
do not apply because they use considerations specific to subgroups of
$\gsp_2(\mathbb F_p)$, while the genus $3$ techniques of \cite{zywina2015example} use results specific
to $\mathbb Q$, such as Serre's conjecture.
However, while there do exist curves over $\mathbb Q$ of every genus $g \geq 3$ whose Jacobian
has Galois representation with image equal to $\gsp_{2g}(\widehat{\mathbb Z})$ by \cite[Theorem 1.1]{landesman-swaminathan-tao-xu:lifting-symplectic-group},
there are no such curves over $\mathbb Q$ of genus $1$ or $2$ by \cite[Proposition 2.5]{zywina2015example}.
Therefore, in order to provide the desired example, we will need techniques applying over number fields $K \neq \mathbb Q$.
It is known that there {\em exist} curves of genus $2$ with Galois representation image equal to
$\gsp_{2g}(\widehat{\mathbb Z})$ over every number field $K \neq \mathbb Q$ so that $K \cap \mathbb Q^{\cyc} = \mathbb Q$ where $\mathbb Q^{\cyc}$ is the maximal
cyclotomic extension of $\mathbb Q$, as follows from \cite[Theorem 1.1]{landesman-swaminathan-tao-xu:rational-families}.
However, the proof there is non-constructive, and so does not lead to any concrete examples.

	There are several examples of curves of genus $2$ whose associated Galois representations have large image:
in~\cite[Theorem 5.4]{dieulefait:explicit-determination-of-the-images},
	Dieulefait gives an example of a
	genus-$2$ curve over $\mathbb Q$
	whose Jacobian has mod-$\ell$
	image equal to $\gsp_4(\mathbb Z/\ell \mathbb Z)$
	for $\ell \ge 5$, and in \cite[Theorem 1.3]{landesman-swaminathan-tao-xu:hyperelliptic-curves},
	the authors give an example of a genus $2$ curve over $\mathbb Q$ so that the associated Galois representation has image of index $2$ in
$\gsp_4(\widehat{\mathbb Z})$.

As mentioned above, by 
\cite[Proposition 2.5]{zywina2015example}
there are no genus $2$ curves over $\mathbb Q$
whose associated Galois representation has image equal
to
$\gsp_4(\widehat{\mathbb Z})$.
We briefly recall the group-theoretic reason for this:
Since $\rho_{J(C)}(G_{\mathbb Q^{\cyc}}) = \rho_A(G_{\mathbb Q}) \cap \Sp_4(\widehat{\mathbb Z})$,
$\rho_{J(C)}(G_{\mathbb Q^{\ab}}) = [\rho_{J(C)}(G_{\mathbb Q}), \rho_{J(C)}(G_{\mathbb Q})]$,
and $\mathbb Q^{\cyc} = \mathbb Q^{\ab}$, we obtain
\begin{align*}
    \rho_A(G_{\mathbb Q}) \cap \Sp_4(\widehat{\mathbb Z}) = [\rho_{J(C)}(G_{\mathbb Q}), \rho_{J(C)}(G_{\mathbb Q})].
\end{align*}
If there were a curve $C$ over $\mathbb Q$ with
$\rho_{J(C)}(G_{\mathbb Q}) = \gsp_4(\widehat{\mathbb Z})$,
the above would imply that the commutator of $\gsp_4(\widehat{\mathbb Z})$ contains all of $\Sp_4(\widehat{\mathbb Z})$.
However, this is false, as can even be checked $\mod 2$
because $\gsp_4(\mathbb Z/2\mathbb Z) \simeq \Sp_4(\mathbb Z/2 \mathbb Z) \simeq S_6$,
which has commutator of index $2$.

The group theoretic obstruction to adelic surjectivity of Galois representations of genus $2$ curves 
from 
\cite[Proposition 2.5]{zywina2015example} 
described
above disappears over number fields $K$ with $K^{\cyc}$ of even index in $K^{\ab}$.
Despite this, prior to this paper, we could not find any examples in the literature of genus $2$ curves over nontrivial extensions $K/\mathbb Q$ with adelic image equal to all of $\gsp_4(\widehat{\mathbb Z})$. The critical new ingredient that enables our explicit construction of a curve $C$ whose associated Galois representation is surjective comes from ~\cite{anni-dokchitser:hyperelliptic}, where Anni and Dokchitser give strong control over the image of the mod $\ell$ representations in terms of the reduction of $C$ at various primes of $\mathscr{O}_K$. Using these techniques, we obtain the following result:
\begin{theorem}\label{theorem:polynomial} 
Let $K=\Q(\alpha)$, where $\alpha^3+\alpha+1=0$
and let $C$ be the genus $2$ hyperelliptic curve which is the regular projective completion of the affine curve $y^2=f(x)$, where $f(x)\in \OO_K[x]$ is the polynomial given by
\begin{align*}f(x):= \hspace{.1cm} &x^6 -1255129022x^5 + 213499328x^4 - 739544064x^3 -1479402560x^2 \\&+938024640x - 486022320+ 85534400\alpha+54644800\alpha^2.\end{align*}
Then $\rho_{J(C)}(G_K)=\gsp_4(\widehat{\Z})$.
\end{theorem}
The remainder of the paper is devoted to proving \autoref{theorem:polynomial}.
We now outline its proof.
In \autoref{adicsection}, 
we reduce the problem of computing $\rho_{J(C)}(G_K)$ 
to showing 
$\rho_{J(C),\ell}(G_K)\supseteq \mathrm{Sp}_4(\Z/\ell\Z)$. In \autoref{modlsection} we apply the results of \cite{anni-dokchitser:hyperelliptic}
to give a criterion 
to show $\rho_{J(C),\ell}(G_K)\supseteq \mathrm{Sp}_4(\Z/\ell\Z)$ for all primes $\ell$ not in the finite set $\left\{ 2,3,5,17 \right\}$. Finally, in
\autoref{proofsection}, we verify the conditions of the criterion from \autoref{modlsection}
and then check that $\rho_{J(C),\ell}(G_K)\supseteq \mathrm{Sp}_4(\Z/\ell\Z)$
at each of the remaining primes $\ell \in \left\{ 2,3,5,17 \right\}$.

\subsection*{Acknowledgements}
We would like to thank Aaron Greicius and David Zureick-Brown for originally suggesting the problem and for providing comments on the manuscript.
We also thank Brian Conrad and Jackson Morrow for helpful comments.
We thank 
Samuele Anni,
Vladimir Dokchitser,
Zev Rosengarten,
Jesse Silliman,
Ashvin Swaminathan, James Tao, and Yujie Xu for beneficial conversations. Special thanks is due to the anonymous referee for their exceptionally attentive suggestions.
This material is based upon work supported by the 
National Science Foundation Graduate Research Fellowship Program under Grant No.~DGE-1656518.

\section{Reducing the problem of adelic surjectivity}\label{adicsection}
In this section, for $C$ the curve from \autoref{theorem:polynomial}, we reduce the problem of showing that 
$\rho_{J(C)}(G_K)=\gsp_4(\widehat{\Z})$ to verifying
$\rho_{J(C),\ell}(G_K)\supseteq\Sp_4(\Z/\ell\Z)$ for all primes $\ell$.
This is accomplished in \autoref{lemma:modl}.
The key result in attaining this reduction is \autoref{lemma:reduction}, an 
analogue of \cite[Theorem 3.1]{greiciusthesis} for $\gsp_4(\widehat{\Z})$ in place of $\gsp_2(\widehat{\Z})$. 

Before stating \autoref{lemma:reduction}, we introduce some notation.
From the identification $\gsp_4(\widehat{\Z})=\prod_{\ell}\gsp_4(\Z_\ell)$, denote by $\pi_\ell\colon \gsp_4(\widehat{\Z})\to \gsp_4(\Z_\ell)$ the projection onto the $\ell$-adic factor. 
Let $\mult\colon  \gsp_4(\widehat{\Z}) \rightarrow \widehat{\Z}^\times$ denote the mult map from the definition of $\gsp$. Then we define $\Sp_4(\widehat{\Z}) := \ker(\mult)$. Also, recalling the identification $\gsp_4(\Z/2\Z) \simeq S_6$,
let $\sgn\colon  \gsp_4(\widehat{\Z}) \rightarrow \left\{ \pm 1 \right\}$ denote the composition of the reduction $\bmod$-$2$
$\Phi_2\colon \gsp_4(\widehat{\Z}) \rightarrow \gsp_4(\Z/2\Z)$ with the usual sign map $\gsp_4(\Z/2\Z) \simeq S_6 \rightarrow \left\{ \pm 1 \right\}$.
\begin{lemma}\label{lemma:reduction} Let $H\subseteq \gsp_4(\widehat{\Z})$ be a closed subgroup such that: \begin{enumerate} \item $\pi_\ell(H)\supseteq\mathrm{Sp}_4(\Z_\ell)$ for all $\ell$.
	\item The map $(\sgn, \mult)\colon  \gsp_4(\widehat{\Z}) \rightarrow \left\{ \pm 1 \right\} \times \widehat{\Z}^\times$ is surjective when restricted to $H$.
\end{enumerate}
Then $H=\gsp_4(\widehat{\Z})$.
\end{lemma}
\begin{proof}
Let $G:=[\gsp_4(\widehat{\Z}),\gsp_4(\widehat{\Z})]$ be the derived subgroup of $\gsp_4(\widehat{\Z})$.
By \cite[3.3.6]{omeara1978symplectic} (see also \cite[Lemma 3.4]{landesman-swaminathan-tao-xu:hyperelliptic-curves}), we have $$G=\Phi_2^{-1}(A_6) \cap \Sp_4(\widehat{\Z}).$$ 
Because the kernel of $(\sgn, \mult)$ is precisely $G$, we conclude $(\sgn, \mult)\colon  \gsp_4(\widehat{\Z}) \rightarrow \left\{ \pm 1 \right\} \times \widehat{\Z}^\times$ is
the abelianization map.
Suppose $H \neq \gsp_4(\widehat{\Z})$. Then by \cite[Lemma 2.2]{greiciusthesis} 
we may assume that $H$ is a maximal closed subgroup. Since the $\mathrm{mult}$ map is surjective, condition (1) 
implies that $\pi_\ell(H)=\gsp_4(\Z_\ell)$. By \cite[Lemma 2.3]{landesman-swaminathan-tao-xu:rational-families} the factors $\gsp_4(\Z_\ell)$ 
have no finite simple nonabelian quotients in 
common. 
Hence, \cite[Proposition 2.5]{greiciusthesis} implies that the image of $H$ in the abelianization $\{\pm 1\}\times \widehat{\Z}^\times$ 
is a proper subgroup. This contradicts (2), and so we must in fact have
$H = \gsp_4(\widehat{\Z})$.
\end{proof}
In order to verify (2) above, we record the following useful criterion, whose proof is completely analogous to that given in
\cite[Theorem 3.1]{greiciusthesis}.
\begin{lemma}
	\label{lemma:sgn-mult-surjection}
	Suppose $D$ is a hyperelliptic curve which is the regular projective completion of the affine curve $y^2 = h(x)$ defined over a number field $L$, which is
	degree $3$ over $\mathbb Q$.
	Then, $(\sgn, \mult) \circ \rho_{J(D)}\colon  G_L \rightarrow \left\{ \pm 1 \right\} \times \widehat{\Z}^\times$
	is surjective if 
	$L \cap \mathbb Q^{cyc}=\mathbb Q$ and $ \disc h$ is not of the form $k^2 q$ for $k \in L, q\in \mathbb Q$.
\end{lemma}
\begin{proof}
	First, we show $\mult \circ \rho_{J(D)}$ is surjective when $L \cap \mathbb Q^{\cyc}=\mathbb Q$.
Recall that the symplectic form on $J(D)[\ell]$ is the Weil pairing, so the composition $\mathrm{mult}\circ \rho_{J(D)}$ is identified with the cyclotomic character. Then since $L\cap \Q^{\mathrm{cyc}}=\Q$, the composition $\mathrm{mult}\circ \rho_{J(D)}$ is surjective because the cyclotomic character is surjective.

We next wish to show the joint map $\left( \sgn, \mult \right)$ is surjective. Given that $\mult$ is surjective,
to show $\left( \sgn, \mult \right)$ is surjective, we claim it suffices to verify
$\sqrt{\disc(h)} \notin L^{\cyc},$ for $L^{\cyc}$ the compositum $L \mathbb Q^{\cyc}$.
Indeed, because the two-torsion of $J(D)$ is generated by differences of Weierstrass points of $D$,
for $\sigma \in G_L$, $\sgn(\sigma)= 1$ if and only if $\sigma$ acts as an even permutation on the $6$ Weierstrass points of $D$.
As $\sqrt{\disc(h)}$ is a multiple of the differences of the Weierstrass points, $\sgn(\sigma) = 1$ if and only if $\sqrt{\disc(h)}$ is fixed by $\sigma$.
So, in order to show $(\sgn, \mult) \circ \rho_{J(D)}\colon  G_L \rightarrow \left\{ \pm 1 \right\}\times \widehat{\Z}^\times$ is jointly surjective, it suffices to show the kernel of $(\sgn,\mult) \circ \rho_{J(D)}$ is strictly contained in $\ker(\mult \circ \rho_{J(D)}) = \mathrm{Gal}(\overline{L}/L^{\cyc})$. 
Since $\ker(\sgn \circ \rho_{J(D)}) = \mathrm{Gal}(\overline{L}/L(\sqrt {\disc h}))$, we only need verify $\sqrt{\disc h} \notin L^{\cyc}$.

To conclude the proof, we only need to show that if $\disc (h)$
is not of the form $k^2 q$ for $k \in L$ and $q \in \mathbb Q$, then $\sqrt{\disc (h)} \notin L^{\cyc}$.
Indeed, any quadratic extension of $L$ contained in $L^{\cyc}$ is necessarily of the form $L(\sqrt{q})$ for $q \in \mathbb Q$.
Therefore, if $L(\sqrt{\disc (h)}) \subset L^{\cyc}$ we obtain $\sqrt{\disc (h)}\in L(\sqrt q)$ for some $q \in \mathbb Q$.
This implies $\sqrt{\disc (h)}=a + b \sqrt q$ for $a,b \in L$. Since $\disc (h) \in L$, and $L$ has degree $3$ over $\mathbb Q$, if $\sqrt q \in L$,
we must have $\sqrt q \in \mathbb Q$.
This yields either $a=0$ or $b =0$, and so
$\disc (h) = k^2 q$ for $k \in L, q \in \mathbb Q$.
\end{proof}

We now use \autoref{lemma:reduction} and \autoref{lemma:sgn-mult-surjection}
to recover the behavior of the adelic representation from the mod-$\ell$ representations:
\begin{lemma}\label{lemma:modl} Let $C$ be the curve defined in \autoref{theorem:polynomial}, and suppose the associated mod-$\ell$ representations satisfy $\rho_{J(C),
	\ell}(G_K)\supseteq\Sp_4(\Z/\ell\Z)$ for all primes $\ell$. Then $\rho_{J(C)}(G_K)=\gsp_4(\widehat{\Z})$.
\end{lemma}
\begin{proof}

	By \cite[Theorem B]{weigel:on-the-profinite-completion-of-arithmetic-groups-of-split-type} (see also \cite[Theorem 1.3]{vasiu2003surjectivity} and \cite[Theorem 1]{landesman-swaminathan-tao-xu:lifting-symplectic-group})
no proper subgroup of $\mathrm{Sp}_4(\Z_\ell)$ can surject onto $\mathrm{Sp}_4(\Z/\ell\Z)$ under reduction mod $\ell$, so the assumption that $\rho_{J(C),\ell}(G_K)\supseteq \Sp_4(\Z/\ell\Z)$ implies that $\pi_\ell(\rho_{J(C),\ell^\infty}(G_K)) \supseteq \Sp_4(\Z_\ell)$. 

Hence, by \autoref{lemma:reduction}, in order to complete the proof, we only need verify that the $(\sgn,\mult)$ map is surjective.
By \autoref{lemma:sgn-mult-surjection}, it suffices to check
$K \cap \mathbb Q^{\cyc} =\mathbb Q$ and $\sqrt {\disc f}$ is not of the form $k^2 q$ for $k \in K, q \in \mathbb Q$.

First, we check
$K \cap \mathbb Q^{\cyc} =\mathbb Q$.
Suppose for the sake of contradiction that 
$K \cap \mathbb Q^{\cyc} \neq \mathbb Q$.
Because $[K:\mathbb Q]=3$ is prime, 
$K \cap \mathbb Q^{\cyc} \neq \mathbb Q$ implies $K\cap \mathbb Q^{\cyc}=K$. 
This would imply $K/\mathbb Q$ is an abelian extension and hence Galois,
contradicting that $K$ is not Galois over $\mathbb Q$. 

To conclude the proof, we only need to check $\disc (f)$ is not of the form $k^2 q$ for $k \in K, q \in \mathbb Q$.
Indeed, in $K$, $(3)$ factors as $(3)=\p_3\mathfrak{q}_3$ with $\p_3 \neq \mathfrak q_3$ and $\p_3 \mid (\mathrm{disc}(f))$, $\p_3^2\nmid (\mathrm{disc}(f))$, and $\mathfrak{q}_3\nmid \mathrm{disc}(f)$. So, if $\mathrm{disc}(f)=k^2q$ for some $k\in K$ and $q\in \Q$, comparing the exponents of primes dividing $(3)$, we get $\p_3=(\p_3^a\mathfrak{q}_3^b)^2(\p_3\mathfrak{q}_3)^c$ for some integers $a,b,c$.
Comparing powers of $\mathfrak{q}_3$ yields $-2b=c$, so $c$ is even. 
However, by comparing powers of $\p_3$, 
this would imply $2a + c$ is even and also equal to $1$, a contradiction.\end{proof}

It therefore remains to show that the image of the representations $\rho_{J(C),\ell}$ contain $\mathrm{Sp}_4(\Z/\ell\Z)$ for all $\ell$.
\section{Controlling the mod-$\ell$ representations}\label{modlsection}

A sufficient condition for surjectivity at odd primes is given in \autoref{surjective}. To state it, we first define the relevant terminology.
\begin{definition}\label{definition:primitive} Let $V$ be a symplectic vector space over a field $k$, and let $G$ be a subgroup of $\gsp(V)$. We say that $\{V_1,\ldots, V_k\}$ is a \textit{non-trivial} $G$-\textit{stable decomposition} of $V$ into symplectic subspaces if the $V_i$ are proper symplectic subspaces $V_i \subset V$ with $V = \bigoplus_{i=1}^kV_i$, the symplectic pairing is non-degenerate on $V_i$, and there is a homomorphism $\phi \colon G\to S_k$ such that $\sigma(V_i)=V_{\phi(\sigma)(i)}$ for $\sigma \in G$. If no such decomposition exists, $V$ is said to be \textit{primitive}.
\end{definition}
\begin{definition}\label{definition:transvection} An element $\sigma\in \gsp(V)$ is called a \textit{transvection} if $\sigma$ is unipotent (has all eigenvalues equal to $1$) and $\sigma-I$ has rank $1$. \end{definition}
\begin{theorem}[\protect{\cite[Theorem 1.1]{hall2008big}, \cite[Proposition 2.2]{zywina2015example}}]\label{surjective}
Let $\ell$ be an odd prime. If the mod-$\ell$ representation $\rho_{J(C),\ell}(G_K)\subseteq \gsp_4(\Z/\ell\Z)$ 
of $G_K$ on the $\mathbb Z/\ell \mathbb Z$ vector space $J(C)[\ell]$
is irreducible, primitive, and contains a transvection, then $\rho_{J(C),\ell}(G_K)\supseteq \mathrm{Sp}_4(\Z/\ell\Z)$. 
\end{theorem}

The results of \cite{anni-dokchitser:hyperelliptic} give explicit congruence conditions on $f(x)$ so that the criteria of the above theorem are satisfied at all but a finite set of primes $\ell$. For its statement and proof we require two further definitions:

\begin{definition}[\protect{\cite[Definition 1.2, Definition 1.3]{anni-dokchitser:hyperelliptic}}]
For a prime ideal $\p$ of $\OO_K$ with residue characteristic $p$ and corresponding valuation $v_\p$, let $F$ denote the completion of $K$ at $v_\p$, viewed as an extension of $\Q_p$, and let $\OO_F$ denote the ring of integers.
A polynomial $f(x)=x^n+a_{n-1}x^{n-1}+\ldots+a_0\in \OO_F[x]$ is said to be {\it $t$-Eisenstein} at $\p$ if $v_\p(a_i)\geq t$ for $1\leq i\leq n-1$ and $v_\p(a_0)=t$. 

We say that a monic, squarefree polynomial 
$f(x)\in \OO_F[x]$ has {\it type $t-\{q_1,\ldots,q_k\}$} at $\p$ for rational primes $q_1,\ldots,q_k$ if it can be factored as $$
f(x)=h(x)\prod_{i=1}^kg_i(x-\alpha_i)$$
for some $\alpha_i\in \OO_F$ and $h(x),g_i(x)\in \OO_F[x]$ such that $\alpha_i \not \equiv \alpha_j \mod \p$ for all $i \neq j$, $g_i(x)$ is a $t$-Eisenstein polynomial of degree $q_i$, $h(x)$ is separable mod $\p$, and
$h(\alpha_i) \not \equiv 0 \mod \p$ for all $i$. 
We say some $f \in \OO_K[x]$ 
is $t$-Eisenstein (respectively of type $t - \left\{ q_1, \ldots, q_k \right\}$)
if the image of $f$ in $\OO_F[x]$ 
is $t$-Eisenstein
(respectively of type $t - \left\{ q_1, \ldots, q_k \right\}$).
\end{definition}

Note that the following definitions concern vector spaces over $\overline{\mathbb F}_\ell$, whereas the other representation-theoretic considerations in this section, such as Theorem 3.3, deal with $\mathbb F_\ell$-vector spaces.

\begin{definition}[\protect{\cite[Definition 4.6]{anni-dokchitser:hyperelliptic}}]
For $\p \subset \OO_K$ a prime, let $I_\p \subset G_K$ denote the inertia group at $\p$.
We will say that $f(x) \in \OO_K[x]$ is
$\ell$-\textit{admissible} at $\p$ if for every $G_K$-stable decomposition $J[\ell] \otimes \overline{\mathbb F}_\ell = \bigoplus_{i=1}^kV_i$
into symplectic $\overline{\mathbb F}_\ell$-subspaces, $I_{\p}$ acts trivially on $\{V_1, \ldots , V_k\}$. 
We will say that $f(x) \in \OO_K[x]$ is \textit{admissible} at $\p$ if it is $\ell$-admissible at $\p$ for
every odd prime number $\ell$ not divisible by $\p$.
\end{definition}

The following theorem is immediate upon combining the results of \cite{anni-dokchitser:hyperelliptic}.
We spell out the details for completeness.
\begin{theorem}\label{AD} Let $K$ be a number field with no nontrivial unramified extensions (possibly excepting the infinite places), $f(x)\in \OO_K[x]$ a monic irreducible polynomial of degree
$2g+2$, and $\ell>g$ a rational prime, such that the following conditions are satisfied:\begin{enumerate}
\item There exist rational primes $q_1,q_2,q_3$ such that $q_1\leq q_2<q_3 < 2g+2$ and $q_1+q_2=2g+2$.
\item There exist primes $\p_{t_1}$ and $\p_{t_2}$ of distinct, odd residue characteristics such that $f(x)$ has type $1-\{2\}$ at $\p_{t_1}$ and $\p_{t_2}$.
\item There exists a prime $\p_2$ of odd residue characteristic $p_2$ such that the order of the residue field $\F_{\p_2}$ at $\p_2$ is a primitive root mod $q_1$ and $q_2$ and $f(x)$ has type $1-\{q_1,q_2\}$ at $\p_2$.
\item There exists a prime $\p_3$ of odd residue characteristic $p_3$ such that the order of the residue field $\F_{\p_3}$ at $\p_3$ is a primitive root mod $q_3$ and $f(x)$ has type $2-\{q_3\}$ at $\p_3$.
\item The curve $C$ defined by $y^2=f(x)$ has good reduction at all primes above $2$.
\item The curve $C$ has semistable reduction at all primes $\p\nmid 2\p_2\p_3$.
\item For all primes $\p \mid \ell$ we have $\ell> 2e_\p+1$, where $e_\p$ is the ramification degree of $\p$.
\item We have $\ell\not\in \{q_1,q_2,q_3,p_2,p_3\}$.
\end{enumerate} Then we have $\rho_{J(C),\ell}(G_K)\supseteq \mathrm{Sp}_{2g}(\Z/\ell\Z)$.
\end{theorem}\begin{proof}
Let $\ell>g$ be a rational prime satisfying conditions (7) and (8). It suffices by \autoref{surjective}
to show that $\rho_{J(C),\ell}(G_K)$ is irreducible, primitive, and contains a transvection. 
By condition (2), \cite[Lemma 2.9]{anni-dokchitser:hyperelliptic} implies that $\rho_{J(C),\ell}(G_K)$ 
contains a transvection. Note that the residue characteristic of at least one of the $\p_{t_i}$ will be distinct from $\ell$, as required in \cite[Lemma 2.9]{anni-dokchitser:hyperelliptic}. By conditions (1), (3), (4), and (8), \cite[Lemma 3.2]{anni-dokchitser:hyperelliptic} 
implies that $\rho_{J(C),\ell}$ is irreducible. Then since $\rho_{J(C),\ell}$ is irreducible and $K$ has no nontrivial extensions unramified at all finite places, 
\cite[Proposition 4.4]{anni-dokchitser:hyperelliptic} reduces the problem of showing that $\rho_{J(C),\ell}$ 
is primitive to showing that $f(x)$ satisfies the two conditions of \cite[Proposition 4.7]{anni-dokchitser:hyperelliptic}:
that $f(x)$ is admissible at all $\p \nmid \ell$ and that $f(x)$ is $\ell$-admissible at all $\p \mid \ell$.

We first check that $f(x)$ is admissible at all $\p$, which is the first condition of \cite[Proposition 4.7]{anni-dokchitser:hyperelliptic}.
Conditions (5) and (6), together with \cite[Lemma 7.5(ii)]{anni-dokchitser:hyperelliptic}, imply that $J(C)$ is semistable at all $\p\neq \p_2,\p_3$, so that $f(x)$ is admissible at all $\p\neq \p_2,\p_3$ by 
\cite[Lemma 4.9]{anni-dokchitser:hyperelliptic}. Then note that the primitive root assumption of condition (3) implies that $q_1,q_2\neq p_2$, so that $f(x)$ is admissible at $\p_2$ and $\p_3$ by 
\cite[Lemmas 4.10 and 4.11]{anni-dokchitser:hyperelliptic}, respectively. 
So, we have verified the first condition of \cite[Proposition 4.7]{anni-dokchitser:hyperelliptic}.

To complete the proof, we verify the second condition of \cite[Proposition 4.7]{anni-dokchitser:hyperelliptic},
i.e., $f(x)$ is $\ell$-admissible at $\p$ for all $\p \mid \ell$. 
By \cite[Proposition 4.12]{anni-dokchitser:hyperelliptic}, it suffices to check that
$\mathrm{disc}(f)\not\in \p^2$ (guaranteeing semistability at $\p$) and $\ell>\mathrm{max}(g,2e_{\p}+1)$, 
where $e_\p$ is the ramification degree of $\p$. 
The first statement follows from conditions (6) and (8), and the second statement follows from condition (7).
\end{proof}

\section{Verifying the example}\label{proofsection}

Using \autoref{AD}, we can now compute the $\bmod$-$\ell$ image of the Galois representation associated to our hyperelliptic curve $C$. We first note that
$K$ has no nontrivial unramified extensions, again considering only the finite places, which essentially follows
from Minkowski's bound on the discriminant of an extension of $\mathbb Q$.
\begin{lemma}[\protect{\cite{26504}}]\label{lemma:unramified} Let $K=\Q(\alpha)$, where $\alpha$ is a root of $x^3+x+1$. Then $K$ has no nontrivial extensions unramified at all finite places.\end{lemma}
Next, we apply \autoref{AD} to verify surjectivity of our Galois representation
at all but a finite set of primes.
\begin{lemma}\label{lemma:applyingAD} The mod-$\ell$ Galois representations $\rho_{J(C),\ell}$ associated to the curve $C$ in the statement of \autoref{theorem:polynomial} satisfy $\rho_{J(C),\ell}(G_K)\supseteq \Sp_4(\Z/\ell\Z)$ for all $\ell\not\in \{2,3,5,17\}$.
\end{lemma}
\begin{proof}
	We apply \autoref{AD} to the $f(x)\in \OO_K[x]$ in \autoref{theorem:polynomial}, taking $q_1=q_2=3$, $q_3=5$, $\p_{t_1}=(7)$, $\p_{t_2}=(3,\alpha+2)^2$, $\p_2=(5)$, and $\p_3=(17,\alpha+6)$, where we check that $\#\F_{(5)}= 125\equiv 2 \mod 3$ and $\#\F_{(17,\alpha+6)}=17\equiv 2\mod 5$ are primitive roots. These choices of the $q_i$ and $p_i$, along with the assumptions of the lemma, are immediately seen to satisfy conditions (1) and (8). 

	We next verify condition (7). Note that for all primes $\p$ we have $2e_\p+1\leq 2[K:\Q]+1=7$, so that the condition is trivially satisfied for all $\ell>7$. Then since $\ell\neq 2,3,5$ by assumption, it only remains to check the case $\ell=7$, and since $7$ is unramified (and even inert) in $K$, we have $7 > 3 = 2e_{(7)} + 1$, so the inequality is satisfied. 
	
	By construction, $f(x)$ satisfies the following congruence conditions
\begin{align}
     f(x)&\equiv (x^2+7)(x^4+1) &\mod &(7)^2\\
     f(x)&\equiv (x^2+3)(x^4-2x^3+2x^2+1) &\mod &(3,\alpha+2)^2\\
     f(x)&\equiv (x^3+5)((x+1)^3+5) &\mod &(5)^2\\
     f(x)&\equiv (x^5+17^2)(x+1) &\mod &(17,\alpha+6)^3\\
     f(x)&\equiv x^6+2x^5+2^4 &\mod &2^6.
\end{align}
So, conditions (2)-(4) of \autoref{AD} are satisfied, and condition (5) follows from the final congruence condition by \cite[Lemma 7.7]{anni-dokchitser:hyperelliptic}. 

To conclude, we verify condition (6).
By \cite[Lemma 7.5(i)]{anni-dokchitser:hyperelliptic} in order to show $C$ is semistable at $\p$ it suffices to check $f(x)$ has no roots of multiplicity greater than $2$ over an algebraic closure of the residue field at $\p$. Therefore, it suffices to verify $\p \nmid \mathrm{GCD}(\mathrm{disc}(f),\mathrm{disc}(f'))$. 
A \texttt{magma} calculation shows that the only $\p$ for which $\p \nmid \mathrm{GCD}(\mathrm{disc}(f),\mathrm{disc}(f'))$ 
are $\p=(2),(5),(17,\alpha+6)$,
so condition (6) holds. 
Thus \autoref{AD} shows $\rho_{J(C),\ell}(G_K)\supseteq \mathrm{Sp}_4(\Z/\ell\Z)$ for all $\ell\not\in \{2,3,5,17\}$. 
\end{proof}

It remains only to check that $\rho_{J(C),\ell}(G_K) \supseteq \mathrm{Sp}_4(\Z/\ell\Z)$ at the remaining primes $2,3,5,$ and $17$.
\begin{lemma}\label{lemma:mod2} For $f(x)\in \OO_K[x]$ as in the statement of \autoref{theorem:polynomial}, we have $\rho_{J(C),2}(G_K)=\Sp_4(\Z/2\Z)$.\end{lemma}
\begin{proof}For $\ell=2$ we have $\gsp_4(\F_2)=\mathrm{Sp}_4(\F_2)=S_6$, and we can identify the $2$-torsion points of $J(C)$ with differences of 
Weierstrass points.
Since the Weierstrass points correspond to 
roots of $f$, the $G_K$ action on $J(C)[2]$ is determined by the Galois group of $f$. A \texttt{magma} calculation shows the Galois group of $f$ is $S_6$.
Since $S_6 \simeq \Sp_4(\Z/2\Z)$, 
$\rho_{J(C),2}(G_K)= \mathrm{Sp}_4(\Z/2\Z)$.
\end{proof}
\begin{lemma}\label{lemma:frobenius} For $f(x)\in \OO_K[x]$ as in the statement of \autoref{theorem:polynomial}, we have $\rho_{J(C),\ell}(G_K)\supseteq \Sp_4(\Z/\ell\Z)$ for all $\ell\in \{3,5,17\}$.\end{lemma}\begin{proof}
	For this verification, we use \autoref{surjective}. Since $f(x)$ has type $1-\{2\}$ at two distinct odd primes, ~\cite[Lemma 2.9]{anni-dokchitser:hyperelliptic} implies that $\rho_{J(C),\ell}(G_K)$ contains a transvection for all $\ell$. To show irreducibility and primitivity of $J(C)[\ell]$ for $\ell\in\{3,5,17\}$ we use Frobenius elements at primes of good reduction to show the non-existence of $G_K$-stable decompositions (as vector spaces over $\mathbb F_\ell$). 
First, note that	
if the characteristic polynomial $P_\p=\mathrm{det}(TI-\rho_{J(C),\ell^\infty}(\mathrm{Frob}_\p))\in \Z[T]$ is irreducible mod $\ell$ then $J(C)[\ell]$ must be irreducible as a $G_K$-module.
Further, if $P_\p$ is irreducible and $\mathrm{tr}(\mathrm{Frob}_\p)\not \equiv 0 \mod \ell$ 
we claim $J(C)[\ell]$ must be primitive. 
Indeed, if there were some decomposition $J(C)[\ell] =\bigoplus_{i=1}^k V_i$ with $k > 1$ and all $V_i$ proper subspaces so that the $V_i$ are permuted by
the action of $G_K$, then $\mathrm{tr}(\mathrm{Frob}_\p)\not \equiv 0 \mod \ell$ implies some $V_j$ must be fixed by Frobenius.
This contradicts irreducibility of $P_\p$.

For each $\ell \in \left\{ 3,5,7 \right\}$, it therefore suffices to find a prime $\p$ with $P_\p$ irreducible and $\mathrm{tr}(\mathrm{Frob}_\p)\not \equiv 0 \mod \ell$.
Calculating 
the characteristic polynomials of various primes in \texttt{magma},
we find that for $\ell=3,5$ we can take $\p=(37,\alpha+12)$ and for $\ell=17$ we can take $\p=(29,\alpha+3)$, where the characteristic polynomials are given by:
\begin{align*}
P_{(37,\alpha+12)}&=T^4 + 16T^3 + 136T^2 + 592T + 1369
\\
P_{(29,\alpha+3)}&=T^4 - 5T^3 + 48T^2 - 145T + 841.
\end{align*}
Thus $\rho_{J(C),\ell}(G_K)\supseteq \mathrm{Sp}_4(\F_\ell)$ for all $\ell$, as desired. \end{proof}
Our main theorem now follows immediately:
\begin{proof}[Proof of \autoref{theorem:polynomial}]
	Combining \autoref{lemma:applyingAD}, \autoref{lemma:mod2}, and \autoref{lemma:frobenius} we obtain $\rho_{J(C),\ell}(G_K)\supseteq \Sp_4(\F_\ell)$ for all primes $\ell$. 
	By \autoref{lemma:modl}, we then have $\rho_{J(C)}(G_K)=\gsp_4(\widehat{\Z})$, completing the proof.
\end{proof}

\bibliographystyle{alpha}
\bibliography{galois-surface-bibliography}

\end{document}